\documentclass[12pt]{amsart}
\usepackage{amsfonts}
\usepackage{amssymb, amscd, amsthm}
\usepackage{latexsym}
\usepackage{multirow}
\usepackage{amsmath}
\usepackage[all]{xy}
\usepackage{mathrsfs}
\usepackage{amsbsy}
\usepackage{amsfonts}
\usepackage{mathrsfs}
\usepackage{verbatim}
\usepackage{version}
\usepackage{color}

\newtheorem{theorem}{Theorem}[section]
\newtheorem{lemma}[theorem]{Lemma}
\newtheorem{thm}[theorem]{Theorem}
\newtheorem{prop}[theorem]{Proposition}
\newtheorem{rem}[theorem]{Remark}
\newtheorem{coro}[theorem]{Corollary}

\newtheorem{con/que}[theorem]{Conjecture/Question}
\newtheorem{que}[theorem]{Question}

\newcommand{\ra}{\rightarrow}
\newcommand{\mo}{\mathcal{O}}
\newcommand{\mf}{\mathcal{F}}

\newcommand{\ma}{\mathcal{A}}
\newcommand{\mb}{\mathcal{B}}
\newcommand{\me}{\mathcal{E}}

\newcommand{\ms}{\mathcal{S}}
\newcommand{\mt}{\mathcal{T}}
\newcommand{\mh}{\mathcal{H}}
\newcommand{\mm}{\mathcal{M}}

\newcommand{\bm}{\mathbf{M}}

\newcommand{\mc}{\mathcal{C}}

\newcommand{\bcoh}{\mathbf{Coh}}

\newcommand{\Aut}{\operatorname{Aut}}
\newcommand{\Hom}{\operatorname{Hom}}
\newcommand{\Ext}{\operatorname{Ext}}

\newcommand{\Ker}{\operatorname{Ker}}
\newcommand{\Tor}{\operatorname{Tor}}
\newcommand{\Ima}{\operatorname{Im}}

\newcommand{\Tot}{\operatorname{Tot}}

\def\<{\langle}
\def\>{\rangle}

\newcommand{\ls}{|L|}
\newcommand{\p}{\mathbb{P}}

\begin{document}
\fontsize{12pt}{14pt} \textwidth=14cm \textheight=21 cm
\numberwithin{equation}{section}
\title{Sheaves on non-reduced curves in a projective surface.}
\author{Yao Yuan}
\thanks{The author is supported by NSFC 21022107 and NSFC 11771229.  }

\begin{abstract}Sheaves on non-reduced curves can appear in moduli space of 1-dimensional semistable sheaves over a surface, and moduli space of Higgs bundles as well.  We estimate the dimension of the stack $\bm_{X}(nC,\chi)$ of pure sheaves supported at the non-reduced curve $nC~(n\geq 2)$ with $C$ an integral curve on $X$.  We prove that the Hilbert-Chow morphism $h_{L,\chi}:\mm^H_X(L,\chi)\ra \ls$ sending each semistable 1-dimensional sheaf to its support have all its fibers of the same dimension for $X$ Fano or with trivial canonical line bundle and $\ls$ contains integral curves.  

~~~

\end{abstract}

\maketitle
\tableofcontents
\section{Introduction.}

\subsection{Motiviations.} Sheaves on non-reduced curves can appear in two types of moduli stacks $\mm^H_X(L,\chi)$ and $\mm_{C,D}^{Higgs}(n,\chi)$ as follows.
\begin{enumerate}
\item[(A)] $\mm^H_X(L,\chi)$ parametrizes semistable sheaves $\mf$ with respect to the polarization $H$ on a projective surface $X$, supported on a curve in the linear system $\ls$ and of Euler characteristic $\chi$.  We have a Hilbert-Chow morphism
\begin{equation}\label{HCmap}h_{L,\chi}:\mm^H_X(L,\chi)\ra \ls,~~\mf\mapsto \text{supp}(\mf).\end{equation}
In general $\ls$ contains singular curves and also curves with non-reduced components.

~~

\item[(B)] $\mm_{C,D}^{Higgs}(n,\chi)$ parametrizes semistable Higgs bundles $(\me,\Theta)$ with respect to the effective divisor $D$ on a smooth curves $C$ with $\me$ rank $n$ and Euler characteristic $\chi$.  We have the Hitchin fibration
\begin{equation}\label{HF}h_{D,\chi}:\mm_{C,D}^{Higgs}(n,\chi)\ra \bigoplus_{i=1}^n H^0(C,\mo_C(iD)),~~(\me,\Theta)\mapsto \text{char}(\Theta).\end{equation}
Denote by $\Tot(\mo_C(D))$ the total space of the the line bundle $\mo_C(D)$.  Let $p:\Tot(\mo_C(D))\ra C$ be the projection, then every Higgs bundle $(\me,\Theta)$ on $C$ gives a 1-dimensional pure sheaf $\mf_{\me}$ on $\Tot(\mo_C(D))$ via the following exact sequence
\[0\ra p^*\me\xrightarrow{\lambda\times p^*\text{id}_{\me}-p^*\Theta}p^*\me(D)\ra \mf_{\me}\ra 0. \]
We naturally have that $\text{supp}(\mf_{\me})$ is defined by the equation $\text{char}(\Theta)$ where $\lambda$ is the variable on the fiber of $p$.  Hence the fiber of the Hitchin fibration over $\lambda^n$, which is also called the central fiber of the Hitchin fibration, consists of sheaves (with some semistablity) on the non-reduced curve in $\Tot(\mo_C(D))$ defined by $\lambda^n$.

Let $X:=\mathbb{P}(\mo_C(D)\oplus\mo_C)$ be the ruled surface over $C$.  Then $X$ is projective and the central fiber of $h_{D,\chi}$ consists of sheaves (with some semistablity) on the non-reduced curve $nC_e$ with $C_e$ the section satisfying $C_e^2=\text{deg}(D)$ and $C_e.K_X=-\text{deg}(D)+2g_C-2$.
\end{enumerate}

People have to consider the dimension of each fibers of $h_{L,\chi}$ in $(A)$ or $h_{D,\chi}$ in $(B)$, if they want to study the flatness of the fibrations,  or if they want to compute the cohomology of sheaves (constructible or coherent) over the moduli space $\mm^H_X(L,\chi)$ ($\mm_{C,D}^{Higgs}(n,\chi)$, resp.) via the fibration $h_{L,\chi}$ ($h_{D,\chi}$, resp.).  

The fibers of $h_{L,\chi}$ ($h_{D,\chi}$, resp.) over integral supports are their compactified Jacobians and hence are of equal dimension denoted by $N_L$ ($N_{C,D,n}$, resp.) which only depends on $\ls$ ($(C,D,n)$, resp.).  But the fibers over non-integral supports are much more complicated, among which the worst are fibers over non-reduced curves.  People expect those fibers would also have dimension $N_L$ or $N_{D,C,n}$ respectively, but in principle they can be of larger dimension. 
So we pose the following question.
\begin{que}\label{mque}Whether all fibers of $h_{L,\chi}$ in $(A)$ ($h_{D,\chi}$ in $(B)$, resp.) are of the expected dimension $N_{L}$ ($N_{C,D,n}$, resp.)?
\end{que}

Some results are already known for Question \ref{mque}.  For instance, Ginzburg showed in \cite{Gin} that the central fiber of $h_{D,\chi}$ is of expected dimension $N_{C,D,n}$ for $D$ the canonical divisor of $C$;  Chaudouard and Laumon showed 
 in \cite{CL} that the central fiber of $h_{D,\chi}$ is of dimension $\leq N_{C,D,n}$ for $\text{deg}(D)>2g-2$; and finally Maulik and Shen showed in \cite{MS} that the fibers of $h_{L,\chi}$ are of the expected dimension $N_L$ for $X$ a toric del Pezzo surface and $L$ ample effective on $X$.  There seems to be no more general result.  Especially the lack of the estimate of the dimension of all fibers of $h_{L,\chi}$ prevents the main result in \cite{MS} from generalizing to all del Pezzo surfaces not necessarily toric (see the paragraph after Remark 0.2 in \cite{MS}).

In this paper we give a complete answer to Question \ref{mque}.  

\subsection{Notations \& Conventions.} All through the paper, let $X$ be a projective surface over an algebraically closed field $\Bbbk$.  Let $C$ be an integral curve on $X$.  Let $\delta_C\in H^0(X,\mo_X(C))$ be the section defining $C$.  Denote by $nC$ the non-reduced curve with multiplicity $n$ over $C$, i.e. the 1-dimensional closed subscheme of $X$ defined by $\delta_C^n$.

Denote by $\overline{\bm_X}(nC,\chi)$ the stack of all 1-dimensional sheaves with schematic supported $nC$ and Euler characteristic $\chi$, and let $\bm_X(nC,\chi)\subset\overline{\bm_X}(nC,\chi)$ be the substack consisting of 1-dimensional pure sheaves.  Then we have $h^{-1}_{\mo_X(nC),\chi}(nC)\subset \bm_X(nC,\chi)$ where $h_{\mo_X(nC),\chi}$ is the Hilbert-Chow morphism defined in (\ref{HCmap}).

Because the stack $\mathbf{Vect}_C(n,\chi)$ of rank $n$ Euler characteristic $\chi$ vector bundles over $C$ is a substack of $\bm_X(nC,\chi)$, we have $$\dim\overline{\bm_X}(nC,\chi)\geq\dim \bm_X(nC,\chi)\geq \dim\mathbf{Vect}_C(n,\chi)=n^2(g_C-1).$$

We use $K_X$ to denote both the canonical divisor of $X$ and the canonical line bundle as well.  For any two (not necessarily integral) curves $C',C''$, we write $C'.K_X$ the intersection number of the divisor class of $C'$ with $K_X$, $C'.C''$ the intersection number of the divisor classes of these two curves and $C'^2:=C'.C'$.  

Denote by $g_{C'}$ the arithmetic genus of $C'$.  We have $g_{C'}-1=\frac{C'^2+C'.K_X}2.$

For a sheaf $\mf$ on $X$, let $\mf(\Sigma):=\mf\otimes\mo_X(\Sigma)$ for $\Sigma$ a curve or a divisor class.

Let $K(X),K(C)$ be the Grothendieck groups of coherent sheaves on $X$ and $C$ respectively.  Let $\star$ stands for $X$ or $C$.  Denote by $\chi_{\star}(-,-):K(\star)\times K(\star)\ra \mathbb{Z}$ the bilinear integral form on $K(\star)$ such that for every two coherent sheaves $\ms,\mt$ on $\star$,
$$\chi_{\star}([\ms],[\mt])=\sum_{j\geq 0}(-1)^j\dim\Ext^j_{\mo_{\star}}(\ms,\mt).$$

\subsection{Results \& Applications.}Our main result is the following theorem.
\begin{thm}[See Theorem \ref{main4}]\label{main1}For any integral curve $C\subset X$, we have
$$\dim \bm_X(nC,\chi)\begin{cases}\leq \frac{n^2C^2}2+\frac{nC.K_X}2=g_{nC}-1,\text{ if }C.K_X\leq 0\\  \\ =\frac{n^2C^2}2+\frac{n^2C.K_X}2=n^2(g_{C}-1),\text{ if }C.K_X>0\end{cases}.$$
\end{thm}
Theorem \ref{main1} suggests that in order to have a positive answer to Question 1.1, it is reasonable to ask the surface $X$ in $(A)$ to be Fano or with $K_X$ trivial and to ask the divisor $D$ in $(B)$ to satisfy $\text{deg}(D)\geq 2g-2$.  

\begin{coro}\label{mcoro}Let $C'=\displaystyle{\prod_{j=1}^s}n_jC_j$ be a curve in the linear system $\ls$ on $X$ with $C_j$ pairwise distinct integral curves.  If $C_j.K_X\leq 0$ for $j=1,\cdots,s$, then 
\[\dim h_{L,\chi}^{-1}(C')\leq g_{C'}-1=N_{L}.\]
In particular, if $X$ is Fano or with $K_X$ trivial, and if either $\ls$ contains integral curves, or $H^1(\mo_X)=H^1(L)=0$ and $\mm^H_X(L,\chi)^s$, the substack of $\mm^H_X(L,\chi)$ consisting of stable sheaves, is not empty, then the Hilbert-Chow morphism $h_{L,\chi}$ in (\ref{HCmap}) has all fibers the expected dimension $g_{\ls}-1$ with $g_{\ls}$ the arithmetic genus of any curve in $\ls$. 
\end{coro}
\begin{proof}Every sheaf in $h_{L,\chi}^{-1}(C')$ can be realized as an excessive extension of pure sheaves on $n_jC_j, ~j=1,\cdots,s$.  Since $C_j$ are pairwise distinct integral curves, for sheaves $\mf_i\in \bm_X(n_jC_j,\chi_j)$ we have 
\[\Hom_{\mo_X}(\mf_i,\mf_j)=\Ext^2_{\mo_X}(\mf_i,\mf_j)=0,~\dim\Ext^1_{\mo_X}(\mf_i,\mf_j)=-\chi_X(\mf_i,\mf_j),~\forall~i\neq j.\]

Hence we have
\begin{eqnarray}\dim h_{L,\chi}^{-1}(C')&\leq& \sum_{j=1}^s\dim\bm_{X}(n_jC_j,\chi_j)-\sum_{i<j}\chi_X(\mf_i,\mf_j)\nonumber\\
&=&\sum_{j=1}^s\left(\frac{n_j^2C_j^2}2+\frac{n_jC_j.K_X}2\right)+\sum_{i<j}n_in_jC_i.C_j\nonumber\\
&=&\frac{(C')^2}2+\frac{C'.K_X}2=g_{C'}-1=N_L.\end{eqnarray}

Let $X$ be Fano or with $K_X$ trivial.  If $\ls$ contains integral curves, then by semicontinuity every fiber of $h_{L,\chi}$ is of dimension no less than $g_{\ls}-1$.  If $H^1(\mo_X)=H^1(L)=0$,  then $\mm^H_X(L,\chi)^s\neq \emptyset$ is smooth of dimension $g_{\ls}-1+\dim\ls$.  Every fiber of $h_{L,\chi}$ is a closed subscheme of $\mm^H_X(L,\chi)$ defined by $\dim\ls$ equations, hence of dimension no less than $g_{\ls}-1$.  We have proved the corollary.
\end{proof}

With Corollary \ref{mcoro}, we can generalize Theorem 0.1 in \cite{MS} to all del Pezzo surfaces.  For the reader's convenience, we write the explicit statement as follows.
\begin{thm}[Generalization of Theorem 0.1 in \cite{MS} to all del Pezzo surfaces]\label{ms}
~Let $X$ be a del Pezzo surface \emph{not necessarily toric} with polarization $H$, and let $L$ be an ample curve class on $X$.  Let $M^H_X(L,\chi)$ be the coarse moudli space of 1-dimensional semistable sheaves with schematic supports in $\ls$ and Euler characteristic $\chi$.  Then we have for any $\chi,\chi'\in\mathbb{Z}$, there are isomorphisms of graded vector spaces
\[\text{IH}^*(M^H_X(L,\chi))\cong \text{IH}^*(M^H_X(L,\chi')),\]
where $\text{IH}^*(-)$ denotes the intersection cohomology.  Moreover, those isomorphisms respect perverse and Hodge filtrations carried by these vector spaces. 
\end{thm}
We refer to \cite{MS} for the details of the proof of Theorem \ref{ms} while Proposition 2.6 in \cite{MS} can be extended to $X$ any del Pezzo surface by Corollary \ref{mcoro}.

Another application of Corollary \ref{mcoro} is the following theorem which generalizes Theorem 1.1 in \cite{Yuan4}.
\begin{thm}\label{genY}Let $X=\p^2$ and $H$ the hyperplan class on $X$.  Let $M_{\p^2}(d,\chi)$ be the coarse moudli space of 1-dimensional semistable sheaves with schematic supports in $|dH|$ and Euler characteristic $\chi$.  Then for $0\leq i\leq 2d-3$ we have 
\[\begin{cases}b_i^v(M_{\p^2}(d,\chi))=0,\text{ for }i\text{ odd }\\
b_i^v(M_{\p^2}(d,\chi))=b_{i}^v(X^{[\frac{d(d-3)}2-\chi_0]}),\text{ for }i\text{ even }\\
h^{p,i-p}_v(M_{\p^2}(d,\chi))=0,\text{ for }p\neq i-p\\
h^{p,i-p}_v(M_{\p^2}(d,\chi))=h^{p,i-p}_v(X^{[\frac{d(d-3)}2-\chi_0]}),\text{ for }i=2p
\end{cases}\]
where $b_i^v$ ($h^{p,q}_v$, resp.) denotes the $i$-th virtual Betti number ($(p,q)$-th virtual Hodge number, resp. ),  $X^{[n]}$ is the Hilbert scheme of $n$-points on $X$ and finally $\chi_0\equiv \chi~(d)$ with $-2d-1\leq \chi_0\leq -d+1$.
\end{thm}
Theorem \ref{genY} follows immediately from Theorem 6.11 in \cite{Yuan4} and Corollary \ref{mcoro} helps us to improve the estimate of the codimension of the subscheme of $M_{\p^2}(d,\chi)$ consisting of sheaves supported at non-integral curves.  Notice that by \cite{Bou}, $h^{p,q}_v(M_{\p^2}(d,\chi))=0$ for any $p\neq q$. 

By Theorem \ref{genY} both the virtual Betti number $b_i^v(M_{\p^2}(d,\chi))$ and the virtual Hodge number $h^{p,q}_v(M_{\p^2}(d,\chi))$ stabilize as $d\ra\infty$.  Write down the generating function
\[Z(t,q):=\sum_{d\geq 0} q^d\left(\sum_{i=0}^{2\dim M_{\p^2}(d,\chi)}b_i^v(M_{\p^2}(d,\chi))t^i\right).\]
Then the coefficient of $t^i$ in $(1-q)Z(t,q)$ is a polynomial in $q$.  However, whether $Z(t,q)$ is a rational function is still a wildly open question.

\subsection{Acknowledgements.} I would like to thank L. G\"ottsche for leading me to study 1-dimensional sheaves over surfaces when I was his PhD student.  I also would like to thank Junliang Shen for answering my questions on their paper \cite{MS}.  
\section{Filtrations for sheaves on non-reduced curves.}
For any $\mf\in \overline{\bm_X}(nC,\chi)$, we can describe $\mf$ via two filtrations as in the following proposition.
\begin{prop}\label{2Fil}Let $\mf\in \overline{\bm_X}(nC,\chi)$, then there are two filtrations of $\mf$:
\begin{enumerate}
\item[(1)]The so-called \textbf{the lower filtration of $\mf$}:
\[0=\mf_0\subsetneq \mf_1\subsetneq\cdots\subsetneq \mf_l=\mf,\]
such that $Q_i:=\mf_{i}/\mf_{i-1}$ are coherent sheaves on $C$ with rank $t_i$.  $\sum t_i=n$, and moreover there are injections $f^i_{\mf}:Q_{i}(-C)\hookrightarrow Q_{i-1}$ induced by $\mf$ for all $2\leq i\leq l$. 

\item[(2)]The so-called \textbf{the upper filtration of $\mf$}:
\[0=\mf^0\subsetneq \mf^1\subsetneq\cdots\subsetneq \mf^m=\mf,\]
such that $R_i:=\mf^{i}/\mf^{i-1}$ are coherent sheaves on $C$ with rank $r_i$.  $\sum r_i=n$, and moreover there are surjections $g^i_{\mf}:R_{i}(-C)\twoheadrightarrow R_{i-1}$ induced by $\mf$ for all $2\leq i\leq m$. 
\end{enumerate}
Moreover we have
\begin{enumerate}
\item[(i)] $l=m$.
\item[(ii)] $\forall~1\leq i\leq m,$ $t_i=r_{m-i+1}$.
\end{enumerate}
\end{prop}

Proposition \ref{2Fil} is not difficult to prove.  One sees easily that $\mf_i=\Ker(\mf\xrightarrow{.\delta_C^i}\mf(iC))$ and $\mf^i=\Ima(\mf((i-m)C)\xrightarrow{.\delta_C^{m-i}}\mf)$.  In particular we have $Q_i\cong \mf_i/\Tor^1_{\mo_X}(\mf_i,\mo_{(i-1)C}((i-1)C))$ and $R_i\cong \mf^i\otimes\mo_C$.  The reader can figure out the proof of Proposition \ref{2Fil} by him/her-self or look at Proposition 5.7, Proposition 5.10 and Lemma 5.11 in \cite{Yuan4} for more details.

\begin{rem}\label{sup}Use the same notations as in Proposition \ref{2Fil}, we can see that $m\leq n$ and $m=\max\{k|\mf\xrightarrow{.\delta_C^{k-1}}\mf((k-1)C)\text{ is not zero}\}$.
\end{rem}

Recall that $K(C)$ is the Grothendieck group of coherent sheaves on $C$.  For any class $\beta\in K(C)$, denote by $r(\beta)$ ($\chi(\beta)$, resp.) the rank (Euler characteristic, resp.) of $\beta$.

Obviously both upper and lower filtrations are uniquely determined by $\mf$.  Hence we can stratify $\overline{\bm_X}(nC,\chi)$ by the filtration type.
Let $\overline{\bm_X}(nC,\chi)^{\beta_1,\cdots,\beta_m}$ ($\overline{\bm_X}(nC,\chi)_{\beta_1,\cdots,\beta_m}$, resp.) with $\beta_1,\cdots,\beta_m\in K(C)$ be the substack of $\overline{\bm_X}(nC,\chi)$ consists of sheaves with upper (lower, resp.) filtrations satisfying $[R_i]=\beta_i$ ($[Q_i]=\beta_i$, resp.).  Notice that $\overline{\bm_X}(nC,\chi)^{\beta_1,\cdots,\beta_m}$ ($\overline{\bm_X}(nC,\chi)_{\beta_1,\cdots,\beta_m}$, resp.) is not empty only if $r(\beta_i)\geq r(\beta_{i-1})$ ($r(\beta_i)\leq r(\beta_{i-1})$, resp.) for all $2\leq i\leq m$.  

Define
\[\overline{\bm_X}(nC,\chi)^{r_1,\cdots,r_m}:=\coprod_{\substack{r(\beta_i)=r_i\\ 1\leq i\leq m}}\overline{\bm_X}(nC,\chi)^{\beta_1,\cdots,\beta_m},\]
\[\overline{\bm_X}(nC,\chi)_{r'_1,\cdots,r'_m}:=\coprod_{\substack{r(\beta'_i)=r'_i\\ 1\leq i\leq m}}\overline{\bm_X}(nC,\chi)_{\beta'_1,\cdots,\beta'_m};\]
\[\bm_X(nC,\chi)^{r_1,\cdots,r_m}:=\bm_X(nC,\chi)\cap \overline{\bm_X}(nC,\chi)^{r_1,\cdots,r_m};\]
\[\bm_X(nC,\chi)_{r'_1,\cdots,r'_m}:=\bm_X(nC,\chi)\cap\overline{\bm_X}(nC,\chi)_{r'_1,\cdots,r'_m}.\]
By Proposition \ref{2Fil} we have
\[\overline{\bm_X}(nC,\chi)^{r_1,\cdots,r_m}=\overline{\bm_X}(nC,\chi)_{r_m,\cdots,r_1};\]
$$\bm_X(nC,\chi)^{r_1,\cdots,r_m}=\bm_X(nC,\chi)_{r_m,\cdots,r_1}.$$
The following lemma is straightforward.
\begin{lemma}\label{exfil}Let $\mf\in\overline{\bm_X}(nC,\chi)^{r_1,\cdots,r_m}$, and let $\mf',\mf''$ lie in the following two exact sequences over $X$
\[0\ra \mf\ra\mf'\ra T'\ra 0,~~0\ra T''\ra \mf''\ra \mf\ra0 ;\]
where $T',T''$ are zero-dimensional sheaves of $\mo_{nC}$-module.  Then we have
\[\mf'~(\mf'',~\text{resp.})\in \overline{\bm_X}(nC,\chi'(\chi'',\text{ resp.}))^{r_1,\cdots,r_m},\]
where $\chi'=\chi+\text{length}(T')$ and $\chi''=\chi+\text{length}(T'')$.
\end{lemma}
\begin{rem}\label{tfuF}Use the same notations as in Proposition \ref{2Fil}.  If $\mf\in\bm_X(nC,\chi)$, i.e. $\mf$ is pure, then the factors $Q_i$ of the lower filtration are torsion free over $C$ while the factors $R_i$ of the upper filtration may still contain torsion.
We can take another upper filtration
\[0=\widetilde{\mf}^0\subsetneq \widetilde{\mf}^1\subsetneq\cdots\subsetneq \widetilde{\mf}^m=\mf,\]
such that $\widetilde{R}_i:=\widetilde{\mf}^{i}/\widetilde{\mf}^{i-1}$ are torsion free sheaves on $C$ and every $\widetilde{\mf}^i$ is an extension of some zero dimensional sheaf by $\mf^i$ in the upper filtration.  Actually $\widetilde{R}_i$ is the maximal torison-free quotient of $\widetilde{\mf}_i\otimes\mo_C$ and $r(\widetilde{R}_i)=r(R_i)$ for all $1\leq i\leq m$ by Lemma \ref{exfil}.

There are also morphisms $\widetilde{g}^i_{\mf}:\widetilde{R}_{i}(-C)\rightarrow \widetilde{R}_{i-1}$ induced by $\mf$ for all $2\leq i\leq m$.  But
$\widetilde{g}^i_{\mf}$ is not necessary surjective.  We call this filtration \textbf{the torsion-free upper filtration of $\mf$}.
\end{rem}

\section{The case with reduced curve smooth.}
In this section, we prove our main theorem for $C$ a smooth curve.
\begin{prop}\label{smc}Let $C$ be a smooth curve, then we have
\[\dim \overline{\bm_X}(nC,\chi)^{r_1,\cdots,r_m}\leq \frac{n^2C^2}2+\frac{C.K_X}2\left(\displaystyle{\sum_{j=1}^m}r_j^2\right).\]
In particular, 
$$\dim \overline{\bm_X}(nC,\chi)\begin{cases}\leq \frac{n^2C^2}2+\frac{nC.K_X}2=g_{nC}-1,\text{ if }C.K_X\leq 0\\  \\ =\frac{n^2C^2}2+\frac{n^2C.K_X}2=n^2(g_{C}-1),\text{ if }C.K_X>0\end{cases}.$$
\end{prop}
\begin{proof}
Since the possible choices of $(\beta_1,\cdots,\beta_m)$ such that $r(\beta_j)=r_j,~j=1,\cdots,m$ form a discrete set.  Hence Proposition \ref{smc} follows straightforward from the following lemma.\end{proof}

\begin{lemma}\label{smcb}For every $(\beta_1,\cdots,\beta_m)\in K(C)^m$, we have
\[\dim \overline{\bm_X}(nC,\chi)^{\beta_1,\cdots,\beta_m}\leq \frac{n^2C^2}2+\frac{C.K_X}2\left(\displaystyle{\sum_{j=1}^m}r(\beta_j)^2\right).\]
\end{lemma}
Before proving Lemma \ref{smcb}, we need to define some stacks.  For any $\eta\in K(C)$, denote by $\bcoh_{\eta}$ the stack of coherent sheaves on $C$ of class $\eta$.  Let $\underline{\eta}=(\eta_1,\cdots,\eta_m)\in K(C)^{m}$, denote by $\widehat{\bcoh_{\underline{\eta}}}$ the stack of chains $\mc_{\bullet}$
\[\mc_{\bullet}:=[\xymatrix@C=1.2cm{\ms_m\ar@{->>}[r]^{d_m\quad}&\ms_{m-1}(C)\ar@{->>}[r]^{~d_{m-1}} &\cdots \ar@{->>}[r]^{d_2\qquad}& \ms_1((m-1)C) }],\]
where $\ms_i$ are coherent sheaves on $C$ of class $\eta_i$.
Two chains $\mc_{\bullet},\mc_{\bullet}'$ are isomorphic if we have the following commutative diagram with vertical arrows all isomorphisms
\[\xymatrix@C=1.2cm{\ms_m\ar@{->>}[r]^{d_m\quad}\ar[d]_{\cong} &\ms_{m-1}(C)\ar@{->>}[r]^{~d_{m-1}}\ar[d]_{\cong} &\cdots \ar@{->>}[r]^{d_2\qquad}& \ms_1((m-1)C)\ar[d]_{\cong}\\
\ms'_m\ar@{->>}[r]^{d'_m\quad}&\ms'_{m-1}(C)\ar@{->>}[r]^{~d'_{m-1}} &\cdots \ar@{->>}[r]^{d'_2\qquad}& \ms'_1((m-1)C) }.\]

Let $\widetilde{\bcoh}_{\underline{\eta}}$ be the stack of the pairs $(\mh,\mh_{\bullet})$, where $\mh$ is a coherent sheaf on $C$ of class $\sum_{j=1}^m\eta_j$ and where $\mh_{\bullet}$ is a filtration
\[\mh_1\subset \mh_2\subset\cdots\subset \mh_m=\mh\]
satisfying $[\mh_k]=\sum_{j=1}^k\eta_j$ for $k=1,\cdots,m$.

Define $\gamma_i:=\beta_m-\beta_{m-i}\otimes[\mo_X(C)]$, $\alpha_i:=\gamma_i-\gamma_{i-1}$ (with $\beta_0=\gamma_0=0$) and $\underline{\alpha}=(\alpha_1,\cdots,\alpha_m)$.  We have several natural maps as follows

\begin{equation}\label{nmaps}
\xymatrix{
\bcoh_{\beta_m}\times\cdots\times\bcoh_{\beta_1}&\widehat{\bcoh}_{\underline{\beta}}\ar[r]^{\Phi}\ar[l]_{\qquad\qquad\pi_m}&\widetilde{\bcoh}_{\underline{\alpha}}\ar[d]^{\pi_q}\\
\overline{\bm_X}(nC,\chi)^{\beta_1,\cdots,\beta_m}\ar[u]^{\pi_s}&&\bcoh_{\alpha_1}\times\cdots \times\bcoh_{\alpha_m}},
\end{equation}
where $\Phi$ ($\pi_m$, resp.) is defined by sending $[\xymatrix@C=0.5cm{\ms_m\ar@{->>}[r]^{d_m\quad}&\ms_{m-1}(C)\ar@{->>}[r]^{~d_{m-1}} &\cdots \ar@{->>}[r]^{d_2\qquad}& \ms_1((m-1)C) }]$ to $(\mh=\ms_m,\mh_i=\Ker d_{m-i+1}\circ\cdots\circ d_m)$ ($(\ms_m,\cdots,\ms_1)$, resp.),  $\pi_q$ is defined by sending $(\mh,\mh_{\bullet})$ to $(\mh_1,\mh_2/\mh_1,\cdots,\mh/\mh_{m-1})$, and finally $\pi_s$ is defined by sending $\mf$ to its factors $(R_m,\cdots,R_1)$ of the upper filtration.

\begin{lemma}\label{dimet}Let $\Phi,\pi_q,\pi_m,\pi_s,\underline{\beta},\underline{\alpha}$ be as in (\ref{nmaps}), then we have
\begin{enumerate}
\item[(i)] $\Phi$ is an isomorphism;
\item[(ii)] $\dim \pi_q^{-1}((\mh_1,\mh_2/\mh_1,\cdots,\mh/\mh_{m-1}))=-\displaystyle{\sum_{i<j}}\chi_C(\alpha_j,\alpha_i)$;
\item[(iii)] for every $(\ms_m,\cdots,\ms_1)\in \Ima(\pi_m)$
$$\dim \pi_m^{-1}((\ms_m,\cdots,\ms_1))=\displaystyle{\sum_{i=2}^m}\dim \Hom_{\mo_C}(\ms_i,\ms_{i-1}(C));$$
\item[(iv)] $\Ima(\pi_s)\subset\Ima(\pi_m)$;
\item[(v)] for every $(R_m,\cdots,R_1)\in \Ima(\pi_s)$
$$\dim \pi_s^{-1}((R_m,\cdots,R_1))\leq \displaystyle{\sum_{i<j}}r(\beta_i)r(\beta_j)C^2+\displaystyle{\sum_{i=1}^{m-1}}\dim\Hom_{\mo_C}(R_i,R_{i+1}(K_X)).$$
\end{enumerate}
\end{lemma}
\begin{proof}(i) is obvious.

(ii) is analogous to Proposition 3.1 (ii) in \cite{Sch}, hence we omit the proof here and refer to \cite{Sch}.

(iii) It is easy to see that $\dim \pi_m^{-1}((\ms_m,\cdots,\ms_1))=\displaystyle{\sum_{i=2}^m}\dim \Hom_{\mo_C}(\ms_i,\ms_{i-1}(C))^{sur}$ where $\Hom_{\mo_C}(\ms_i,\ms_{i-1}(C))^{sur}$ is the subset of $\Hom_{\mo_C}(\ms_i,\ms_{i-1}(C))$ consisting of surjective maps.  But according to semicontinuity we have
\[\dim \Hom_{\mo_C}(\ms_i,\ms_{i-1}(C))^{sur}=\dim \Hom_{\mo_C}(\ms_i,\ms_{i-1}(C))\]
for any $(\ms_m,\cdots,\ms_1)\in \Ima(\pi_m)$.

(iv) is also obvous.

We prove (v) by induction on $m$.  For $m=1$ there is nothing to prove.

For $m\geq 2$, we have the commutative diagram
\[\xymatrix{\overline{\bm_X}(nC,\chi)^{\beta_1,\cdots,\beta_m}\ar[r]^{\pi''_s\qquad\qquad}\ar[rd]_{\pi_s}& \bcoh_{\beta_m}\times \overline{\bm_X}((n-r(\beta_m))C,\chi)^{\beta_1,\cdots,\beta_{m-1}}\ar[d]^{\pi'_s\otimes Id_{\bcoh_{\beta_m}}}\\
&\bcoh_{\beta_m}\times\cdots\times\bcoh_{\beta_1}},\]
where $\pi''_s$ is defined by sending $\mf$ to $(R_m,\mf^{m-1})$ with $0=\mf^0\subsetneq \mf^1\subsetneq\cdots\subsetneq \mf^m=\mf$ the upper filtration.
We can get $\mf$ as an extension of $R_m$ by $\mf^{m-1}$
\[0\ra\mf^{m-1}\ra \mf\ra R_m\ra 0.\]

Notice that every element $\sigma\in \Aut_{\mo_X}(\mf)$ induces an element in $\Aut_{\mo_X}(R_m)\times \Aut_{\mo_X}(\mf^{m-1})$ because $R_m\cong \mf\otimes\mo_C$.  Moreover $\Aut_{\mo_X}(R_m)\cong \Aut_{\mo_C}(R_m)$.  Hence we have
\[\Aut_{\mo_X}(\mf)\xrightarrow{f_1}\Aut_{\mo_C}(R_m)\times\Aut_{\mo_X}(\mf^{m-1})\]
and $\Hom_{\mo_X}(R_m,\mf^{m-1})\subset \Ker(f_1)$.  Denote by $\Ext^1_{\mo_X}(R_m,\mf^{m-1})_{\mf}$ the subset of $\Ext^1_{\mo_X}(R_m,\mf^{m-1})$ consists of extensions with middle term $\mf$, then $\Aut_{\mo_C}(R_m)\times\Aut_{\mo_X}(\mf^{m-1})$ acts on $\Ext^1_{\mo_X}(R_m,\mf^{m-1})_{\mf}$ transitively with stabilizer $\Ima(f_1)$.  Therefore
\begin{eqnarray}\label{indest}
\dim (\pi''_s)^{-1}((R_m,\mf^{m-1}))&\leq& \dim \Ext^1_{\mo_X}(R_m,\mf^{m-1})-\dim \Hom_{\mo_X}(R_m,\mf^{m-1})\nonumber\\
&=&-\chi_{X}(R_m,\mf^{m-1})+\dim \Ext^2_{\mo_X}(R_m,\mf^{m-1})\nonumber
\end{eqnarray}
By Serre duality $$\dim \Ext^2_{\mo_X}(R_m,\mf^{m-1})=\dim \Hom_{\mo_X}(\mf^{m-1},R_m(K_X)).$$  Since $R_m$ is a sheaf of $\mo_C$-module,  $\mf^{m-2}\subset \Ker(g),~\forall~g\in \Hom_{\mo_X}(\mf^{m-1},R_m(K_X))$.  Hence
\begin{eqnarray}
\dim \Hom_{\mo_X}(\mf^{m-1},R_m(K_X))&=&\dim \Hom_{\mo_X}(R_{m-1},R_m(K_X))\nonumber\\
&=&\dim \Hom_{\mo_C}(R_{m-1},R_m(K_X))\nonumber
\end{eqnarray}
By Riemann-Roch $\chi_{X}(R_m,\mf^{m-1})=-\displaystyle{\sum_{j=1}^{m-1}}r(\beta_j)r(\beta_m)C^2$. Therefore
\[\dim (\pi''_s)^{-1}((R_m,\mf^{m-1}))\leq \sum_{j=1}^{m-1}r(\beta_j)r(\beta_m)C^2+\dim \Hom_{\mo_C}(R_{m-1},R_m(K_X))\]
On the other hand
\[\dim \pi_s^{-1}((R_m,\cdots,R_1))\leq \dim (\pi'_s)^{-1}((R_{m-1},\cdots,R_1))+\dim (\pi''_s)^{-1}((R_m,\mf^{m-1}))\]
We get (v) by applying induction assumption to $\mf^{m-1}$.
\end{proof}

\begin{proof}[Proof of Lemma \ref{smcb}]As $C$ is smooth, $\dim\bcoh(\alpha)=-\chi_C(\alpha,\alpha)=(1-g_C)r(\alpha)^2$ for any $\alpha\in K(C)$ such that $\bcoh(\alpha)$ is not empty.  Combine (i)-(v) in Lemma \ref{dimet} we have
\begin{eqnarray}\dim\overline{\bm_X}(nC,\chi)^{\beta_1,\cdots,\beta_m}&\leq& \sum_{i<j}r(\beta_i)r(\beta_j)C^2+\sum_{i=1}^{m-1}\dim\Hom_{\mo_C}(R_i,R_{i+1}(K_X))\nonumber\\
&&-\sum_{i\leq j}\chi_C(\alpha_j,\alpha_i)-\sum_{i=2}^m\dim \Hom_{\mo_C}(R_i,R_{i-1}(C))\nonumber
\end{eqnarray}
By Serre duality on $C$, we have $\dim\Hom_{\mo_C}(R_i,R_{i+1}(K_X))=\dim\Ext^1_{\mo_C}(R_{i+1},R_{i}(C))$ since $\mo_C(C+K_X)$ is the canonical line bundle on $C$.  Therefore we have
\begin{equation}\label{es2}\dim\overline{\bm_X}(nC,\chi)^{\beta_1,\cdots,\beta_m}\leq \sum_{i<j}r(\beta_i)r(\beta_j)C^2-\sum_{i\leq j}\chi_C(\alpha_j,\alpha_i)-\sum_{i=2}^m\chi_{C}(R_i,R_{i-1}(C))
\end{equation}
It is easy to see that $\alpha_i=[R_{m-i+1}((i-1)C)]-[R_{m-i}(iC)]$.  Therefore from (\ref{es2}) we have
\begin{eqnarray}\label{es3}\dim\overline{\bm_X}(nC,\chi)^{\beta_1,\cdots,\beta_m}&\leq& \sum_{i<j}r(\beta_i)r(\beta_j)C^2-\sum_{i=1}^m\chi_C(\alpha_i,\alpha_i)-\sum_{i=1}^{m-1}\chi_{C}(R_i,R_i)\nonumber\\
&&-\sum_{i=2}^m(\chi_{C}(R_i,R_{i-1}(C))+\chi_{C}(R_{i-1}(C),R_i))
\end{eqnarray}
Because we have
\[\chi_C(\alpha,\beta)+\chi_C(\beta,\alpha)=2(1-g_C)r(\alpha)r(\beta),~\forall~\alpha,\beta\in K(C),\]
we get Lemma \ref{smcb} from (\ref{es3}) by a direct computation.

\end{proof}
\begin{rem}\label{nofin}Although $\overline{\bm_X}(nC,\chi)$ has finite dimension, it is not of finite type and actually contains infinite many connected components.  In general even $\bm_X(nC,\chi)$ is only locally of finite type.  However, it will be an interesting question to ask whether the p-reduction of $\bm_X(nC,\chi)$ to a finite field $\mathbb{F}^q$ is of finite volume?
\end{rem}
\begin{rem}\label{locfree}If $\mf$ is a locally free sheaf of $\mo_{nC}$-module, then $R_m$ is torsion-free of rank 1 and hence $R_{m-i}\cong R_{m}(-iC)$ for all $i=1,\cdots, m-1$.  Thus we have
\[\chi(\mf)=\sum_{j=1}^m\chi(R_j)=m\chi(R_m)-\frac{m(m-1)}2C^2.\]
Therefore if $m\not|~2\chi$, there is no locally free sheaf of $\mo_{nC}$-module in $\bm_X(nC,\chi)$.
\end{rem}
At the end of the section, we would like to state a result for extensions on $C\cong \p^1$ which generalizes Lemma 5.2 in \cite{Yuan4}.  We won't need Lemma \ref{p1et} in the rest of the paper, the reader who only concerns the main theorem can also skip it.
\begin{lemma}\label{p1et}Let $C\cong \p^1$.  Take an exact sequence on $X$
\begin{equation}\label{mide}0\ra\mo_{C}(s_1)\ra \me\ra\mo_C(s_2)\ra0.\end{equation}
If $s_1<s_2-C^2$, then $\me$ is a locally free sheaf of rank 2 on $C$ and hence splits into direct sum of two line bundles.
\end{lemma}
\begin{proof}We only need to show the following equality for all $s_1<s_2-C^2$
\begin{equation}\label{exd}\dim\Ext^1_{\mo_C}(\mo_C(s_2),\mo_C(s_1))=\dim\Ext_{\mo_X}^1(\mo_C(s_2),\mo_C(s_1)),.\end{equation}
By Serre duality we have LHS $=\dim H^0(\mo_{\p^1}(s_2-s_1-2))=\max\{0,s_2-s_1-1\}$ and
\footnotesize
\begin{eqnarray}\text{RHS }&=&-\chi(\mo_C(s_2),\mo_C(s_1))+\dim\Hom_{\mo_X}(\mo_C(s_2),\mo_C(s_1))+\dim\Ext_{\mo_X}^2(\mo_C(s_2),\mo_C(s_1))\nonumber\\
&=&C^2+\max\{0,s_1-s_2+1\}+\dim\Hom_{\mo_X}(\mo_C(s_1),\mo_C(s_2+K_X)).\nonumber\\
&=&C^2+\max\{0,s_1-s_2+1\}+\max\{0,s_2-s_1+1+C.K_X\}\nonumber\end{eqnarray}
\normalsize
If $C.K_X\geq -1$, then $C^2\leq-1$ and we have
\[\begin{cases}\text{LHS=RHS}=s_2-s_1-1,~\text{if }s_1\leq s_2-1\\
 \text{LHS=RHS}=0,~\text{if }s_2-1<s_1\leq s_2-1-C^2
\end{cases}\]
If $C.K_X\leq -2$, then $C^2\geq0$ and we have $\text{LHS=RHS}=s_2-s_1-1,~\text{if }s_1\leq s_2-1-C^2$.  Hence (\ref{exd}) holds if $s_1<s_2-C^2$.
\end{proof}

\section{Proof of the main theorem.}
In this section we let $C$ be any integral curve on $X$, not necessarily smooth.  We want to prove the following theorem.

\begin{thm}\label{main4}For any integral curve $C\subset X$, we have
\[\dim \bm_X(nC,\chi)^{r_1,\cdots,r_m}\leq \frac{n^2C^2}2+\frac{C.K_X}2\left(\displaystyle{\sum_{j=1}^m}r_j^2\right).\]
In particular, 
$$\dim \bm_X(nC,\chi)\begin{cases}\leq \frac{n^2C^2}2+\frac{nC.K_X}2=g_{nC}-1,\text{ if }C.K_X\leq 0\\  \\ =\frac{n^2C^2}2+\frac{n^2C.K_X}2=n^2(g_{C}-1),\text{ if }C.K_X>0\end{cases}.$$
\end{thm}

We proceed the proof of Theorem \ref{main4} by induction on the arithmetic genus $g_C$ of $C$.  If $g_C=0$, then $C\cong \p^1$ and Theorem \ref{main4} follows immediately from Proposition \ref{smc}.  Let's assume $g_C>0$ and $C$ is not smooth.  Denote by $P$ a singular point of $C$ and $\theta\geq 2$ the multiplicity of $C$ at $P$.   
Let $\widetilde{X}\xrightarrow{f}X$ be the blow-up at $P$ and $\widetilde{C}:=f^{-1}C$.  Then $\widetilde{C}=C_0+\theta E$ with $C_0$ an integral curve with $g_{C_0}<g_{C}$ and $E\cong \p^1$ the exceptional divisors.  We have analogous definitions for stacks $\overline{\bm_{\widetilde{X}}}(n\widetilde{C},\chi)$ and $\bm_{\widetilde{X}}(n\widetilde{C},\chi)$ although $\widetilde{C}$ is not integral.
\begin{prop}\label{inject}Let $C'\subset X$ be any curve not necessarily integral.  Let $\widetilde{C'}=f^{-1}C'$.  For every $\mf\in\bm_X(C',\chi)$, the pull-back $f^*\mf\in\bm_{\widetilde{X}}(\widetilde{C'},\chi)$.  Moreover $f_*f^*\mf\cong \mf$ and $f^*:\bm_X(C',\chi)\ra\bm_{\widetilde{X}}(\widetilde{C'},\chi)$ is injective.
\end{prop}
\begin{proof}Since $\mf$ is pure, we can take a locally free resolution of it
\begin{equation}\label{reof}0\ra\ma\xrightarrow{A}\mb\ra\mf\ra0,\end{equation}
where $\ma,\mb$ are locally free.
Pull back (\ref{reof}) to $\widetilde{X}$ and we get
\[f^*\ma\xrightarrow{f^*A} f^*\mb\ra f^*\mf\ra 0,\]
where $f^*A$ has to be injective because it is generically injective and $f^*\ma$ is a locally free sheaf over $\widetilde{X}$.  So we have a locally free resolution of length 1 for $f^*\mf$
\begin{equation}\label{reopf}0\ra f^*\ma\xrightarrow{f^*A} f^*\mb\ra f^*\mf\ra 0.\end{equation}
Hence $f^*\mf$ has to be pure of dimension 1 and obviously is supported at $f^*(C')=n\widetilde{C'}$.

Push forward (\ref{reopf}) to $X$ and we get
\[0\ra f_*f^*\ma\xrightarrow{f_*f^*A} f_*f^*\mb\ra f_*f^*\mf\ra R^1f_*f^*\ma,\]
where $R^1f_*f^*\ma\cong \ma\otimes R^1f_*\mo_{\widetilde{X}}=0$.  As $R^if_*\mo_{\widetilde{X}}=0,\forall~i>0$ and $f_{*}\mo_{\widetilde{X}}\cong \mo_X$, it is easy to see $f_*f^*\ma\cong \ma, f_*f^*\mb\cong \mb, f_*f^*A=A$ and hence $f_*f^*\mf\cong \mf$.  Also $\chi(f^*\mf)=\chi(R^{\bullet}f_*f^*\mf)=\chi(f_*f^*\mf)=\chi(\mf)$.

On the other hand for any $\mf_1,\mf_2\in \bm_X(C',\chi)$, we have
$$\Hom_{\mo_{\widetilde{X}}}(f^*\mf_1,f^*\mf_2)\cong \Hom_{\mo_{X}}(\mf_1,f_*f^*\mf_2)\cong \Hom_{\mo_{X}}(\mf_1,\mf_2).$$
The proposition is proved.
\end{proof}
We also have the following lemma.
\begin{lemma}\label{flat}For every pure 1-dimensional sheaf $\mf$ over $X$, $\Tor^1_{\mo_X}(\mf,\mo_{\widetilde{X}})=0$.  In particular, any injective map $\mf_1\stackrel{\imath}{\hookrightarrow}\mf_2$ with $\mf_2/\mf_1$ purely of dimension one remains injective after pulled back to $\widetilde{X}$.
\end{lemma}
\begin{proof}Pull back (\ref{reof}) to $\widetilde{X}$ and we get
\[\Tor^1_{\mo_X}(\ma,\mo_{\widetilde{X}})\ra \Tor^1_{\mo_X}(\mf,\mo_{\widetilde{X}})\ra f^*\ma\xrightarrow{f^*A} f^*\mb\ra f^*\mf\ra 0.\]
As we have already seen in the proof of Proposition \ref{inject}, $f^*A$ is injective.  $\Tor^1_{\mo_X}(\ma,\mo_{\widetilde{X}})=0$ by local freeness of $\ma$.  Hence $\Tor^1_{\mo_X}(\mf,\mo_{\widetilde{X}})=0$.
\end{proof}
\begin{lemma}\label{filtype}Let $\mf\in\bm_X(nC,\chi)^{r_1,\cdots,r_m}$, then
$$f^*\mf\otimes \mo_{nC_0}\in\overline{\bm_{\widetilde{X}}}(nC_0,\chi_0)^{r_1,\cdots,r_m}$$
for some suitable $\chi_0$.
\end{lemma}
\begin{proof}For $\mf\in\bm_X(nC,\chi)^{r_1,\cdots,r_m}$ we can take its torsion-free upper filtration $\widetilde{\mf}^{\bullet}$ as in Remark \ref{tfuF}, hence by Lemma \ref{flat} the pull-back of the filtration $f^*\widetilde{\mf}^{\bullet}$ is still a filtration of $f^*\mf$ which generically coincides with the upper filtration of $f^*\mf\otimes\mo_{nC_0}$, hence
we have $f^*\mf\otimes \mo_{nC_0}\in\overline{\bm_{\widetilde{X}}}(nC_0,\chi_0)^{r_1,\cdots,r_m}.$
\end{proof}


\begin{proof}[Proof of Theorem \ref{main4}]By Proposition \ref{inject}, it is enough to estimate the dimension of $f^*(\bm_{X}(nC,\chi))\subset \bm_{\widetilde{X}}(n\tilde{C},\chi)$. For any $\mf\in\bm_X(nC,\chi)^{r_1,\cdots,r_m}$, $f^*\mf$ lies in the following sequence
\begin{equation}\label{decom}0\ra \mf_{0}\ra f^*\mf\ra \mf_E\ra 0,\end{equation}
where $\mf_E$ is the torsion free quotient of $f^*\mf\otimes \mo_{n\theta E}$.
Since $\mf_{0}$ is the extension of the torsion free quotient of $f^*\mf\otimes\mo_{C^0}(-n\theta E)$ by a zero dimensional sheaf, by Lemma \ref{exfil} and Lemma \ref{filtype} we have $\mf_0\in \bm_{\widetilde{X}}(nC_0,\chi_0)^{r_1,\cdots,r_m}$ with some suitable $\chi_0$.

Take the upper filtration $0=\mf^0\subsetneq \mf^1\subsetneq \cdots \subsetneq\mf^m=\mf$.  Then we have
\[f^*\mf\otimes \mo_{i\theta E}\cong f^*(\mf/\mf^{m-i})\otimes\mo_{i\theta E}, \forall~i=1,\cdots,m-1,\]
which is because $$f^*\mf\otimes \mo_{i\theta E}\cong f^*\mf\otimes \mo_{i\widetilde{C}}\otimes \mo_{i\theta E}\cong f^*(\mf\otimes\mo_{iC})\otimes \mo_{i\theta E}.$$

On the other hand the schematic support of $f^*(\mf/\mf^{m-i})$ is $\left(\displaystyle{\sum_{j=0}^{i+1}r_{m-j}}\right)\widetilde{C}+\eta_i E$ and $\eta_i=0$ iff $\mf/\mf^{m-i}$ contains no torsion supported at $P$.  Hence we have
\[f^*\mf\otimes \mo_{n\theta E}\in \overline{\bm_{\widetilde{X}}}(n\theta E,\chi_E)^{r_1^1,\cdots,r_1^{\theta^1},\cdots,r_m^1,\cdots,r_m^{\theta^m}}\]
such that
\[f^*\mf\otimes\mo_{i\theta E}\in \overline{\bm_{\widetilde{X}}}(n\theta E,\chi_E)^{r_{m-i+1}^1,\cdots,r_{m-i+1}^{\theta^{m-i+1}},\cdots,r_m^1,\cdots,r_m^{\theta^m}},\forall~ i=1,\cdots,m.\]
Therefore \begin{equation}\label{ietheta}\theta^j\leq \theta,\forall ~j=1,\cdots,m.\end{equation}  
Moreover we have 
\begin{equation}\label{iesum}\sum_{j=k}^m(\sum_{t=1}^{\theta^j}r_j^t)\geq \theta(\sum_{j=k}^mr_j),k=2,\cdots,m;~\sum_{j=1}^m(\sum_{t=1}^{\theta}r_j^t)= \theta(\sum_{j=1}^mr_j).\end{equation}

Notice that we also have $\mf_E\in \bm_{\widetilde{X}}(n\theta E,\chi_E)^{r_1^1,\cdots,r_1^{\theta^1},\cdots,r_m^1,\cdots,r_m^{\theta^m}}$ by Lemma \ref{exfil}.

Denote by $\Delta$ the set of all $\underline{r}:=(r_1^1,\cdots,r_1^{\theta^1},\cdots,r_m^1,\cdots,r_m^{\theta^m})$ satisfying both (\ref{ietheta}) and (\ref{iesum}).
We have the map between stacks induced by (\ref{decom})
\[f^*(\bm_X(nC,\chi)^{r_1,\cdots,r_m})\xrightarrow{\pi_E}\coprod_{\chi_0} \bm_{\widetilde{X}}(nC_0,\chi_0)^{r_1,\cdots,r_m}\times \coprod_{\chi_E,\underline{r}\in \Delta} \bm_{\widetilde{X}}(n\theta E,\chi_E)^{\underline{r}}.\]

It is easy to see the fiber of $\pi_E$ is of dimension no bigger than $-\chi_X(\mf_{E},\mf_0)=nC_0.n(\theta E)=n^2C_0.\theta E$ since $\Ext^2(\mf_E,\mf_0)=0$.  By applying the induction assumption to $C_0$ and $E$, we get
\begin{eqnarray}\dim \bm_X(nC,\chi)^{r_1,\cdots,r_m}&\leq& n^2C_0.(\theta E)+\frac{n^2C_0^2}2+\frac{C_0.K_{\widetilde{X}}}2\left(\sum_{j=1}^m r_j^2\right)\nonumber\\
&&+\max_{\underline{r}\in \Delta}\left\{\frac{n^2\theta^2 E^2}2+\frac{E.K_{\widetilde{X}}}2\left(\sum_{j=1}^m\left(\sum_{t=1}^{\theta^j}(r_j^t)^2\right)\right)\right\}.\end{eqnarray}
Since $E.K_{\widetilde{X}}=-1<0$, together with the following Lemma \ref{ineq} we get
\begin{eqnarray}\label{final}\dim \bm_X(n C,\chi)^{r_1,\cdots,r_m}&\leq& n^2C_0.(\theta E)+\frac{n^2C_0^2}2+\frac{C_0.K_{\widetilde{X}}}2\left(\sum_{j=1}^m r_j^2\right)\nonumber\\
&&+\frac{n^2\theta^2 E^2}2+\frac{\theta E.K_{\widetilde{X}}}2\left(\sum_{j=1}^m r_j^2\right)\nonumber\\
&=&\frac{n^2 (C_0+\theta E)^2}2+\frac{(C_0+\theta E).K_{\widetilde{X}}}2\left(\sum_{j=1}^m r_j^2\right)\nonumber\\
&=&\frac{n^2 (\widetilde{C})^2}2+\frac{(\widetilde{C}).K_{\widetilde{X}}}2\left(\sum_{j=1}^m r_j^2\right).\end{eqnarray}
By $\widetilde{C}^2=C^2$ and $\widetilde{C}.K_{\widetilde{X}}=\widetilde{C}.(f^*K_X+E)=C.K_X$, we get the theorem.
\end{proof}

\begin{lemma}\label{ineq}Let $r_m\geq r_{m-1}\geq\cdots\geq r_1$ and let $\theta\in \mathbb{Z}_{>0}$.  If we have real numbers
$$r_m^1,r_m^2,\cdots,r_m^{\theta},r_{m-1}^1,\cdots, r_{m-1}^{\theta}, r_{m-2}^{1},\cdots, r_1^{\theta}$$
such that 
\[\sum_{j=k}^m\left(\sum_{t=1}^{\theta}r_j^t\right)\geq \theta\left(\sum_{j=k}^mr_j\right),k=2,\cdots,m;~\sum_{j=1}^m\left(\sum_{t=1}^{\theta}r_j^t\right)= \theta\left(\sum_{j=1}^mr_j\right).\]
Then we have
\[\sum_{j=1}^m\left(\sum_{t=1}^{\theta}(r_j^t)^2\right)\geq \theta\left(\sum_{j=1}^m(r_j)^2\right)\]
and the equality holds iff $r_j^t=r_j$ for all $t=1,\cdots,\theta$ and $j=1,\cdots,m$.  

In particular we can remove zeros in $\{r_j^t\}$ and ask rest of them to be positive, i.e.
\[r_m^1,r_m^2,\cdots,r_m^{\theta^1},r_{m-1}^1,\cdots, r_{m-1}^{\theta^2}, r_{m-2}^{1},\cdots, r_1^{\theta^m}>0\]
with $\theta^i\leq \theta$ for $i=1,\cdots,m$, then 
\[\sum_{j=1}^m\left(\sum_{t=1}^{\theta^j}(r_j^t)^2\right)\geq \theta\left(\sum_{j=1}^m(r_j)^2\right)\]
and the equality holds iff $\theta^j=\theta$ and $r_j^t=r_j$ for all $t=1,\cdots,\theta$ and $j=1,\cdots,m$. 
\end{lemma}
\begin{proof}Let $\epsilon_j^t:=r_j^t-r_j$, then we have
\[\sum_{j=k}^m\left(\sum_{t=1}^{\theta}\epsilon_j^t\right)\geq 0,k=2,\cdots,m;~\sum_{j=1}^m\left(\sum_{t=1}^{\theta}\epsilon_j^t\right)=0.\]
Hence
\begin{eqnarray}
&&\sum_{j=1}^m\left(\sum_{t=1}^{\theta}(r_j^t)^2\right)=\sum_{j=1}^m\left(\sum_{t=1}^{\theta}(r_j+\epsilon_j^t)^2\right)\nonumber\\
&=&\theta\left(\sum_{j=1}^m(r_j)^2\right)+\sum_{j=1}^m\left(\sum_{t=1}^{\theta}(\epsilon_j^t)^2\right)+2\sum_{j=1}^m\left(r_j\left(\sum_{t=1}^{\theta}\epsilon_j^t\right)\right)\nonumber\\
&=&\theta\left(\sum_{j=1}^m(r_j)^2\right)+\sum_{j=1}^m\left(\sum_{t=1}^{\theta}(\epsilon_j^t)^2\right)+2\sum_{k=2}^m\left((r_{k}-r_{k-1})\left(\sum_{j=k}^m\left(\sum_{t=1}^{\theta}\epsilon_j^t\right)\right)\right)\nonumber\\
&\geq &\theta\left(\sum_{j=1}^m(r_j)^2\right),\nonumber
\end{eqnarray}
where the last inequality is because $r_k\geq r_{k-1}$ and $\displaystyle{\sum_{j=k}^m\left(\sum_{t=1}^{\theta}\epsilon_j^t\right)}\geq 0$.  It is easy to see the equality holds iff $\epsilon_j^t=0$ for all $t=1,\cdots,\theta$ and $j=1,\cdots,m.$
\end{proof}
\begin{rem}\label{algcl}As Theorem \ref{main4} only concerns the dimension, the assumption that the base field $\Bbbk$ is algebraically closed can be removed.
\end{rem}

Yao Yuan\\
Beijing National Center for Applied Mathematics,\\
Academy for Multidisciplinary Studies, \\
Capital Normal University, 100048, Beijing, China\\
E-mail: 6891@cnu.edu.cn.
\end{document}